\documentclass[english]{amsart}

\usepackage{amsfonts,amsmath,amssymb,color,amsthm,mathrsfs}
\usepackage[T1]{fontenc}
\usepackage[english]{babel}
\usepackage{color}
\usepackage{hyperref}
\usepackage[T1]{fontenc}
\usepackage[all]{xy}
\usepackage{mathtools}
\newtheorem{theorem}[equation]{Theorem}
\newtheorem{lemma}{Lemma}[section]
\newtheorem{corollary}[lemma]{Corollary}
\newtheorem{proposition}[lemma]{Proposition}

\theoremstyle{definition}

\theoremstyle{remark}
\newtheorem{remark}[lemma]{Remark}
\newtheorem{example}[lemma]{Example}

\newcommand\dashmapsto{\mapstochar\dashrightarrow}

\newcommand\Bir{{\mathrm{Bir}}}

\newcommand\Q{{\mathbb Q}}

\newcommand\End{{\mathrm{Rat}}}
\newcommand\Res{{\mathrm{Res}}}
\newcommand\A{{\mathbb A}}

\newcommand\C{{\mathbb C}}

\newcommand\p{{\mathbb P}}

\newcommand\M{{\mathcal{M}}}
\newcommand\Cal{{\mathcal{C}}}

\newcommand\Proj{{\mathbb P}}
\newcommand\PGL{{\mathrm{PGL}}}

\newcommand\Aut{{\mathrm{Aut}}}

\newcommand\tr{\hbox to 1mm  {${}^t \!  $} }

\newcommand{\nc}{\newcommand}
\nc{\beq}{\begin{equation}}
\nc{\eeq}{\end{equation}}
\nc{\what}[1]{{\,\widehat{\!#1\!}\,}}
\nc{\wtilde}[1]{{\,\widetilde{\!#1\!}\,}}
\nc{\rmE}{{\rm E}}
\nc{\MW}{{Mordell--\kern-.12exWeil}}

\title[Quadratic maps with a marked periodic point of small order]{Moduli
   spaces of quadratic rational maps with a marked periodic point
   of small order}
\thanks{The first-named author gratefully acknowledges support by the
   Swiss National Science Foundation Grant
   "Birational Geometry" PP00P2\_128422 /1.
   The third-named author gratefully acknowledges support by the
   US National Science Foundation under grant DMS-1100511.}

\author{J\'er\'emy Blanc}
\address{J\'er\'emy Blanc, Universit\"{a}t Basel,
   Mathematisches Institut, Rheinsprung $21$, CH-$4051$ Basel, Switzerland.}
\email{jeremy.blanc@unibas.ch}

\author{Jung Kyu Canci}
\address{Jung Kyu Canci, Universit\"{a}t Basel,
   Mathematisches Institut, Rheinsprung $21$, CH-$4051$ Basel, Switzerland.}
\email{jungkyu.canci@unibas.ch}

\author{Noam D. Elkies}
\address{Noam D. Elkies, Department of Mathematics,
   Cambridge, MA 02138, USA.}
\email{elkies@math.harvard.edu}

\begin{document}
\maketitle

\centerline{\today}

\begin{abstract}
The surface corresponding to the moduli space of quadratic endomorphisms
of $\p^1$ with a marked periodic point of order~$n$ is studied.
It is shown that the surface is rational over $\mathbb{Q}$ when
$n\le 5$ and is of general type for $n=6$.

An explicit description of the $n=6$ surface lets us find several
infinite families of quadratic endomorphisms $f: \p^1 \to \p^1$
defined over $\mathbb{Q}$ with a rational periodic point of order~$6$.
In one of these families, $f$ also has a rational fixed point,
for a total of at least $7$ periodic and $7$ preperiodic points.
This is in contrast with the polynomial case, where it is conjectured
that no polynomial endomorphism defined over $\mathbb{Q}$ admits
rational periodic points of order $n>3$.
\end{abstract}

\section{introduction}
A classical question in arithmetic dynamics concerns
periodic, and more generally preperiodic, points of a rational map
(endomorphism) $f\colon \p^1(\Q)\to \p^1(\Q)$.  A point $p$ is said
to be \emph{periodic of order~$n$} if $f^n(p)=p$ and if $f^i(p)\neq p$
for $0<i<n$. A point $p$ is said to be \emph{preperiodic} if there exists
a non-negative integer $m$ such that the point $f^m(p)$ is periodic.

In \cite[Conjecture 2]{FPS}, it was conjectured that if $f$\/ is
a polynomial of degree~$2$ defined over~$\mathbb{Q}$, then $f$\/ admits
no rational periodic point of order $n>3$.  This conjecture, also
called Poonen's conjecture (because of the refinement made in \cite{Poo},
see \cite{HuIn}), was proved in the cases $n=4$ \cite[Theorem 4]{Mor}
and $n=5$ \cite[Theorem 1]{FPS}. Some evidence for $n=6$ is given
in \cite[Section 10]{FPS}, \cite{Sto} and \cite{HuIn}.
The bound $n>3$ is needed because the polynomial maps $f$\/ of degree~$2$
constitute an open set in~$\A^3$, and for any $p_i$ the condition
$f(p_i) = p_{i+1}$ cuts out a hyperplane in this $\A^3$.

In this article, we study the case where $f$\/ is not necessarily
a polynomial but a rational map of degree~$2$.  Here the space
of such maps is an open set in~$\p^5$,
and again each condition $f(p_i) = p_{i+1}$ cuts out a hyperplane.
Hence the analog of Poonen's conjecture would be that
there is no map defined over~$\Q$
with a rational periodic point of order $n>5$. However, we show that
in fact there are infinitely many pairs $(f,p)$, even up to
automorphism of $\p^1$, such that $f\colon \p^1(\Q)\to \p^1(\Q)$
is a rational map of degree $2$ defined over~$\Q$
and $p$ is a rational periodic point of order~$6$.
We do this by studying the structure of the algebraic variety
parametrising such pairs, called as usual the moduli space.

The study of the moduli spaces considered in this article is also
motivated by some more general facts.
For example Morton and Silverman~\cite{MoS} stated
the so-called Uniform Boundedness Conjecture.
The conjecture asserts that for every number field~$K$\/
the number of preperiodic points in $\p^N(K)$
of a morphism \hbox{$\Phi\colon \p^N\to\p^N$} of degree $d \geq 2$
defined over~$K$ is bounded, by a number depending
only on the integers $d,N$\/ and on the degree $D = [K:\Q]$
of the extension $K/\Q$.
It seems very hard to settle this conjecture, even in the case
$(N,D,d) = (1,1,2)$.  As usual, the way to solve a problem in number theory
starts with the study of the geometrical aspects linked to the
problem. The moduli spaces studied in this article are some geometrical
objects naturally related with the Uniform Boundedness Conjecture.

\bigskip

We next give a more precise structure to our moduli spaces.
All our varieties will be algebraic varieties defined over the
field $\Q$ of rational numbers, and thus also over any field of
characteristic zero.

We denote by $\End_d$ the algebraic variety parametrising all
endomorphisms (rational functions) of degree~$d$ of $\p^1$; it
is an affine algebraic variety of dimension $2d+1$. The algebraic
group $\Aut(\p^1)=\PGL_2$  of automorphisms of $\p^1$ acts by conjugation
on $\End_d$.  J.~Milnor~\cite{Mil.1} proved that the moduli space
M$_2(\mathbb{C})={\rm Rat}_2(\mathbb{C})/{\rm PGL}_2(\mathbb{C})$
is analytically isomorphic to $\mathbb{C}^2$.
J.H.\;Silverman~\cite{Sil.1} generalised this result:
for each positive integer $d$,
the quotient space M$_d={\rm Rat}_d/{\rm PGL}_2$ exists
as a geometric quotient scheme over $\mathbb{Z}$
in the sense of Mumford's geometric invariant theory, and
${\rm Rat}_2/{\rm PGL}_2$ is isomorphic to $\mathbb{A}^2_{\mathbb{Z}}$.
More recently, A.\;Levy~\cite{Levy} proved that the quotient space
M$_d$ is a rational variety for all positive integers $d$.

Let $n\ge 1$ be an integer, and let $\widetilde{\M}_d(n)$ be the
subvariety of $\End_d \times (\p^1)^n$ given by the points $(f,p_1,\dots,p_n)$
such that $f(p_i)=p_{i+1}$ for $i=1,\dots,n-1$, $f(p_n)=p_1$ and
all points $p_i$ are distinct (note that here
  $(f,p_1)$ carries the same information as $(f,p_1,...,p_n)$).
The variety $\widetilde{\M}_d(n)$ has dimension $2d+1$.
For $n\ge 2$, $\widetilde{\M}_d(n)$ is moreover affine,
since $\End_d$ is affine and the subset of $(\p^1)^n$ corresponding
to $n$-uples of pairwise distinct points of $\p^1$ is also affine.

The group $\PGL_2$ naturally acts on $\widetilde{\M}_d(n)$, and
M.\;Manes~\cite{Manes} proved that the quotient $\widetilde{\M}_d(n)/\PGL_2$
exists as a geometric quotient scheme.\footnote{In \cite{Manes},
points of "formal period"~$n$ are considered, so our varieties
$\widetilde{\M}_d(n)$ are $\PGL_2$-invariant open subsets of the
varieties called by the same name in \cite{Manes}.}
We will denote by $\M_d(n)$ the quotient surface $\widetilde{\M}_d(n)/\PGL_2$,
which is an affine variety for $n>1$ and $d=2$.

In \cite[Theorem 4.5]{Manes}, it is shown that the surfaces $\M_2(n)$
are geometrically irreducible for every $n>1$ (a fact which is also
true for $n=1$), but not much else is known about these surfaces.

The closed curve $\mathcal{C}_2(n)\subset \M_2(n)$ corresponding
to periodic points of polynomial maps is better known; it is rational
for $n\le 3$, of genus $2$ for $n=4$, and of genus $14$ for $n=5$,
and its genus rapidly increases with~$n$. Bousch studied these curves from
an analytic point of view in his thesis \cite{Bou} and Morton from
an algebraic point of view in \cite{Mor96}. See \cite[Chapter 4]{Sil.2}
for a compendium of the known results on $\mathcal{C}_2(n)$.

Note that $\M_d(n)$ has an action of the automorphism $\sigma_n$
of order~$n$ which sends the class of $(f,p_1,p_2,\dots,p_n)$ to
the class of $(f,p_2,\dots,p_{n},p_1)$. For $n\ge 5$, the quotient
surface $\M_2(n)/\langle\sigma_n\rangle$ parametrises the set of
orbits of size~$n$ of endomorphisms of $\p^1$ of degree $2$ (see
Lemma~\ref{Lemm:ProjectionAk}).  One approach to Poonen's conjecture,
carried out in \cite{FPS,Mor,Sto} and elsewhere, is to study
the quotient curve $\mathcal{C}_2(n)/\langle\sigma_n\rangle$,
which has a lower genus than $\mathcal{C}_2(n)$.

\smallskip

The aim of this article is to understand the geometry of $\M_2(n)$
and $\M_2(n)/\langle\sigma_n\rangle$ for small~$n$, i.e.\ to describe
the birational type of the surfaces and to determine whether
they contain rational points.
Our main result is the following:
\begin{theorem}\label{Thm:Main}$\ $
\begin{enumerate}
\item
For $1\le n\le 5$, the surfaces $\M_2(n)$ and $\M_2(n)/\langle\sigma_n\rangle$
are rational over $\mathbb{Q}$.
\item
The surface $\M_2(6)$ is an affine smooth surface, birational to
a projective surface of general type, whereas $\M_2(6)/\langle\sigma_6\rangle$
is rational over $\mathbb{Q}$.
\item
 The set $\M_2(6)(\Q)$ of $\Q$-rational points of $\M_2(6)$ is
infinite.\end{enumerate}
\end{theorem}

For $n=6$, we also show that $\M_2(6)/\langle \sigma_6^2 \rangle$
is of general type, whereas $\M_2(6)/\langle \sigma_6^3 \rangle$ is rational.

We finish this introduction by detailing the technique we used
to obtain this result, and especially the parts~$(2)$ and $(3)$:

As we show in $\S\ref{MnP5Ak}$, the variety $\M_2(n)$ naturally
embeds into $\p^5\times \A^{n-3}$ for $n\ge 3$, and the projection
to $\A^{n-3}$ yields an embedding for $n\ge 5$. The surface $\M_2(6)$
can thus be viewed, via this technique, as an explicit sextic hypersurface
in $\A^3$.  However, the equation of the surface is not very nice,
and its closure in $\p^3$ has bad singularities (in particular
a whole line is singular). Moreover, the action of $\sigma_6$ on $\M_2(6)$
is not linear. We thus wanted to obtain a better description of $\M_2(6)$.

The variety $\M_2(6)$ embeds into the moduli space
$P_1^6$ of $6$ ordered points of $\p^1$, modulo $\Aut(\p^1) = \PGL_2$.
This variety $P_1^6$ can be viewed in $\p^5$ as the rational
cubic threefold defined by the equations
$$
\begin{array}{ccccccccccccl}
x_0+x_1+x_2+x_3+x_4+x_5=0,\\
x_0^3+x_1^3+x_2^3+x_3^3+x_4^3+x_5^3=0
\end{array}
$$
(see \cite[Example 2, page 14]{DoOr}).
We obtain $\M_2(6)$ as an open subset of the projective surface
in $P_1^6$ given by
$$
\begin{array}{ccccccccccccl}
x_0^3+x_1^2x_2+x_2^2x_3+x_3^2x_1+x_4^2x_5 +x_4x_5^2=0,
\end{array}
$$
with $\sigma_6$ acting by
$[x_0:\dots:x_5] \mapsto [x_0:x_2:x_3:x_1:x_5:x_4].$
The quotient $\M_2(6)/\langle\sigma_6^3\rangle$ can thus be explicitly
computed; using tools of birational geometry we obtain that it
is rational, so $\M_2(6)$ is birational to a double cover of~$\p^2$.
The ramification curve obtained is the union of a smooth
cubic with a quintic having four double points. Choosing coordinates
on $\p^2$ so that the action corresponding to $\sigma_6$ is an
automorphism of order $3$, and contracting some curves
(see Remark~\ref{Rema:Cont} below),
we obtain an explicit description of the surface $\M_2(6)$:
\begin{proposition}\label{Prop:IntroDetailsM6}
$(i)$ The surface $\M_2(6)$ is isomorphic to an open subset of
the quintic irreducible hypersurface $S_6\subset \p^3$
given by
$$ W^2F_3(X,Y,Z)=F_5(X,Y,Z), $$
where
$$
\begin{array}{rcl}
F_3(X,Y,Z)&=&(X+Y+Z)^3+(X^2Z+XY^2+Y\!Z^2)+2XY\!Z,\\
F_5(X,Y,Z)&=&(Z^3X^2+X^3Y^2+Y^3Z^2)-XY\!Z(Y\!Z+XY+XZ),
\end{array}
$$
and the action of $\sigma_6$ corresponds to the restriction of
the automorphism
$$
[W:X:Y:Z]\mapsto [-W:Z:X:Y].
$$

$(ii)$ The complement of $\M_2(6)$ in $S_6$ is the union of $9$
lines and $14$ conics, and is also the trace of an ample divisor
of $\p^3$: the points of $\M_2(6)$ are points $[W:X:Y:Z]\in S_6$
satisfying that $W^2 \neq XY+YZ+XZ$ and that $W^2,X^2,Y^2,Z^2$
are pairwise distinct.

\end{proposition}
\begin{remark}\label{Rema:Cont}
The blow-up of the singular points of $S_6$ gives a birational
morphism $\hat{S}_6\to S_6$, and the linear system $|K_{\hat{S}_6}|$
induces a double covering $\hat{S}_6\to \p^2$, ramified over  a
quintic and a cubic, corresponding to $F_3=0$ and $F_5=0$. One
can also see that the surface $\hat{S}_6$ is a Horikawa surface
since $c_2=46=5(c_1)^2+36$.
\end{remark}
This result is proved below, by directly giving the explicit isomorphism
that came from the strategy described above (Lemma~\ref{Lemm:M6S6}).
The proof is thus significantly shorter than the original derivation
of the formula.

From the explicit description, it directly follows that $S_6$ is
of general type (Corollary~\ref{Coro:Generaltype}). The set of
rational points should thus not be Zariski dense, according to
Bombieri--Lang conjecture. We have however infinitely many rational
points in $\M_2(6)$, which are contained in the preimage in $S_6$
of the rational cubic curve $X^3 + Y^3 + Z^3 = X^2 Y + Y^2 Z + Z^2 X$,
which is again rational, and the preimages in~$S_6$ of the lines
$X=0$, $Y=0$, $Z=0$, which are elliptic curves of rank~$1$ over~$\Q$.
But, for any number field $K$ and for any finite fixed set $S$
of places of $K$ containing all the archimedean ones, the set of
$S$--integral points of $\M_2(6)$ is finite.

We thank Pietro Corvaja for introducing one of us (Canci)
to the subject of this article and for his comments,
and thank Michelle Manes and Michael Zieve for telling one of us
(Elkies) of the preprint that the other two posted on the arXiv
the same day that he spoke on this question at a BIRS workshop.
Thanks also to Igor Dolgachev for interesting discussions during
the preparation of the article.  We are grateful also to
Umberto Zannier for his useful comments.

\section{Preliminaries}
\subsection{The variety $\End_2$}\label{SubSec:End2}$\ $\\
Associating to $(a_0:\dots:a_5)\in \p^5$ the rational map (endomorphism)
of $\p^1$
$$[u:v]\mapsto [a_0 u^2+a_1 uv+a_2 v^2: a_3 u^2+a_4 uv+a_5v^2],$$
the variety $\End_2$ can be viewed as the open subset of $\p^5$
where  $a_0 u^2+a_1 uv+a_2 v^2$ and $a_3 u^2+a_4 uv+a_5v^2$ have
no common roots; explicitly, it is equal to the open subset
of $\p^5$ which is the complement of the quartic hypersurface defined
by the polynomial
$$
\Res(a_0,\dots,a_5)
= a_2^2a_3^2+a_0^2a_5^2-2a_3a_2a_0a_5-a_1a_2a_3a_4-a_4a_1a_0a_5+a_0a_4^2a_2+a_1^2a_3a_5,
$$
where the polynomial $\Res$ is the homogeneous resultant of the
two polynomials $a_0 u^2+a_1 uv+a_2 v^2$ and $a_3 u^2+a_4 uv+a_5v^2$.

\subsection{Embedding $\M_2(n)$ into $\p^5\times \A^{n-3}$}\label{MnP5Ak}$\
$\\
When $n\ge 3$, any element of $\widetilde{\M}_d(n)$ is in the orbit
under $\Aut(\p^1)=\PGL_2$ of exactly one element of the form $$(f,
[0:1], [1:0], [1:1], [x_1:1],\dots, [x_{n-3}:1]),$$ where $f\in
\End_2$ and $(x_1,\dots,x_{n-3})\in \A^{n-3}$.

In particular, the surface $\M_2(n)$ is isomorphic to a locally
closed subset (and hence a subvariety) of $\End_2 \times \A^{n-3}\subset
\p^5\times \A^{n-3}$.
\begin{lemma}\label{Lemm:ProjectionAk}
Viewing $\M_2(n)$ as a subvariety of $\p^5\times \A^{k}$,
where $k=n-3$ as before,
and assuming that $n\ge 5$, the projection $\p^5\times \A^{k}\to
\A^{k}$ restricts to an isomorphism from $\M_2(n)$ with its image,
which is locally closed in $\A^{k}$, and is an affine surface.

The inverse map sends $(x_1,\dots,x_{k})$ to $(a_0:\dots:a_5)$,
where
$$\begin{array}{rcl}
a_0&=&x_1(x_2x_k+x_1-x_2-x_k),\\
a_1&=&x_1(x_k^2-x_k^2x_2-x_2x_1+x_2+x_2x_1^2-x_1^2),\\
a_2&=& x_1x_k(x_kx_2-x_1x_k+x_2x_1-x_2-x_2x_1^2+x_1^2),\\
a_3&=&  x_1(x_2x_k+x_1-x_2-x_k),\\
a_4&=& -x_1x_k^2+x_k^2+x_kx_1^2-x_k+x_1x_k+x_2x_1^2-x_1^2x_kx_2+x_1-2x_1^2,\\
a_5&=&  0.\end{array}$$
\end{lemma}
\begin{proof}

Let $(f,x_1,\dots,x_k)$ be an element of $\M_2(n)\subset \End_2\times
\A^k$. Recall that $f$ corresponds to the endomorphism
 $$f\colon [u:v]\mapsto
[a_0 u^2+a_1 uv+a_2 v^2: a_3 u^2+a_4 uv+a_5v^2].$$

The equalities $f([0:1])=[1:0]$ and $f([1:0])=[1:1]$ correspond
respectively to saying that $a_5=0$ and $a_0=a_3$. Adding the conditions
$f([x_k:1])=[0:1]$, $f([1:1])=[x_1:1]$ and $f([x_1:1])=[x_2:1]$
yields
\begin{equation}\label{EqMa1}\left(\begin{array}{rrrr}
x_k & 1 & x_k^2 & 0\\
1 & 1& 1-x_1 & -x_1\\
x_1 & 1 & x_1^2(1-x_2) & -x_1x_2\end{array}\right)\left( \begin{array}{r}
a_1 \\ a_2\\ a_3 \\ a_4\end{array}\right)=\left( \begin{array}{r}
0\\ 0 \\ 0\end{array}\right).\end{equation}

We now prove that the $3\times 4$ matrix above has rank $3$, if
$(f,x_1,\dots,x_k)\in\M_2(n)\subset \End_2\times \A^k$.

The third minor (determinant of the matrix obtained by removing
the third column) is equal to
$-x_1(x_1-x_k+x_2x_k-x_2)$. Since $x_1\neq 0$, we only have to
consider the case where $x_1=x_k-x_2x_k+x_2$. Replacing this in
the fourth minor, we get $-x_2 (x_k-1)^2 (x_2-1) (x_2-x_2x_k-1+2x_k)$.
Since $x_2,x_k\notin \{0,1\}$, the only case is to study is when
$x_2-x_2x_k-1+2x_k=0$.  Writing $x_k=t$, this yields
$(x_1,x_2,x_k)=(1-t,\frac{1-2t}{1-t},t)$.
The solutions of the linear system (\ref{EqMa1}) are in this case
given by $a_1=-a_3 t+a_4$ and $a_2=-a_4 t$, and yields a map $f$
which is not an endomorphism of degree $2$, since $a_3u+a_4v$ is
a factor of both coordinates.

The fact that the matrix has rank $3$ implies that the projection
yields an injective morphism $\pi\colon\M_2(n)\to \A^k$. It also
implies that we can find the coordinates $(a_0:\dots:a_5)$ of $f$
as polynomials in $x_1,x_2,x_k$. A direct calculation yields the
formula given in the statement. It remains to see that the image
$\pi(\M_2(n))$ is locally closed in $\A^{k}$, and that it is an
affine surface.

To do this, we describe open and closed conditions that define
$\pi(\M_2(n))$. First, the coordinates $x_i$ have to be pairwise
distinct and different from $0$ and $1$. Second, we replace
$x_1,x_2,x_k$ in the formulas that give $a_0,\dots,a_5$, compute
the resultant $\Res(a_0,\dots,a_5)$ (see $\S \ref{SubSec:End2}$)
and ask that this resultant is not zero. These open conditions give the
existence of a unique map $f\in \End_2$ associated to any given
$x_1,x_2,x_k$. We then ask that $f([x_i:1])=[x_{i+1}:1]$ for $i=2,\dots,k-1$,
which are closed conditions. This shows that $\M_2(n)$ is locally
closed in $\A^k$, and moreover that it is an affine surface since
all open conditions are given by the non-vanishing of a finite
set of equations.
\end{proof}
\begin{corollary}
The surface $\M_2(5)$ is isomorphic to an open affine subset of
$\A^2$, and is thus smooth, and rational over $\Q$.
\end{corollary}
\begin{proof}
This result follows directly from Lemma~\ref{Lemm:ProjectionAk}.
\end{proof}
\section{The surfaces $\M_2(n)$ for $n\le 5$}
The proof of the rationality of the surfaces $\M_2(n)$ and $\M_2(n)/\langle
\sigma_n\rangle$ for $n\le 5$ is done by case-by-case analysis.

Note that the rationality of $\M_2(n)$ implies that
$\M_2(n)/\langle \sigma_n\rangle$ is unirational, and thus
geometrically rational (rational over the algebraic closure of $\Q$)
by Castel\-nuovo's rationality criterion for algebraic surfaces.
However, over~$\Q$ there exist rational surfaces with non-rational quotients,
for example some double coverings of smooth cubic surfaces with only one
line defined over $\Q$. The contraction of the line in the cubic gives
a non-rational minimal del Pezzo surface of degree $4$, which admits
a rational double covering \cite[IV, Theorem 29.4]{Manin}.

 Hence, the rationality of $\M_2(n)/\langle \sigma_n\rangle$ is
not a direct consequence of the rationality of $\M_2(n)$.

\begin{lemma}
The surfaces $\M_2(3)$ and $\M_2(4)$ are isomorphic to affine open
subsets of $\A^2$ and are thus smooth, and rational over $\Q$.
The surfaces $\M_2(3)/\langle \sigma_3\rangle$ and $\M_2(4)/\langle
\sigma_4\rangle$ are also rational over $\Q$.
\end{lemma}
\begin{proof}
$(i)$
We embed $\M_2(4)$ in $\End_2\times \A^1\subset \p^5\times \A^{1}$
as in \S\ref{MnP5Ak}.  Recall that an element $(f,x)$ of $\M_2(4)$
corresponds to an endomorphism
$$
f \colon [u:v] \mapsto [a_0 u^2+a_1 uv+a_2 v^2: a_3 u^2+a_4 uv+a_5v^2]
$$
and a point $x\in \A^1$ which satisfy
$$
f([0:1])=[1:0],\ f([1:0])=[1:1],\ f([1:1])=[1:x],\ f([1:x])=[0:1].
$$
This implies  that $[a_0:\dots:a_5]$ is equal to
$$
[-a_1 x-a_2 x^2: a_1: a_2: -a_1 x-a_2 x^2: -a_1x^2-a_2x^3+2a_1x+xa_2+a_2x^2: 0],
$$
so $\M_2(4)$ is isomorphic to an open subset of $\p^1\times \A^1$,
with coordinates $([a_1:a_2],x)$.  Because the resultant of the map equals
$$
-a_2x^2(x-1)(x-2)(a_1+a_2+xa_2)(a_1+xa_2)(a_1x-a_1+a_2x^2),
$$
we obtain an open affine subset of $\A^2$.

$(ii)$ The map $\sigma_4$ sends $(f,[0:1],[1:0],[1:1],[1:x])$ to
$(f,[1:0],[1:1]$, $[1:x],[0:1])$, which is in the orbit of
$$
(gfg^{-1},g([1:0]),g([1:1]),g([1:x]),g([0:1]))
=
(gfg^{-1},[0:1],[1:0],[1:1],[x-1:x]),
$$
where $g\colon [u:v]\mapsto [(1-x)v:x(u-v)]$.
The endomorphism $gfg^{-1}$ corresponds to
$[u:v]\mapsto [b_0 u^2+b_1 uv+b_2 v^2: b_3 u^2+b_4 uv+b_5v^2]$,
where $[b_0:\dots:b_5]$ is equal to
$[(x-1)(xa_2+a_1+a_2)x^2: (x-1)(a_2-a_2x^2-a_1x-xa_2)x: (x-1)^2(a_1+xa_2):
(x-1)(xa_2+a_1+a_2)x^2: -(a_2x^3+a_1x^2-a_2x^2-2a_1x-xa_2+a_1)x: 0]$.
Hence, $\sigma_4$ corresponds to
$$
([a_1:a_2],x) \mapsto
 \left([(a_2-a_2x^2-a_1x-xa_2)x: (x-1)(a_1+xa_2)],\frac{x}{x-1}\right).
$$

The birational map $\kappa\colon \p^1\times \A^1\dasharrow \A^2$
taking $([a_1:a_2],x)$ to

% $$\begin{array}{c}\kappa \colon ([a_1:a_2],x)\dashmapsto \left(\frac{a_2(a_1+(x+1)a_2)x}{(x^3+x^2-x)a_2^2+(x-1)a_1^2+2(x^2-1)a_1a_2}, \frac{(xa_1+a_2x^2-a_1)(a_1+xa_2)}{(x^3+x^2-x)a_2^2+(x-1)a_1^2+2(x^2-1)a_1a_2}\right),\end{array}$$
{\small
$$
\left(
\frac{a_2(a_1+(x+1)a_2)x}{(x^3+x^2-x)a_2^2+(x-1)a_1^2+2(x^2-1)a_1a_2},
\frac{(xa_1+a_2x^2-a_1)(a_1+xa_2)}{(x^3+x^2-x)a_2^2+(x-1)a_1^2+2(x^2-1)a_1a_2}
\right),
$$
}
whose inverse is
$$
\begin{array}{c}\kappa^{-1}\colon (x,y)\dashmapsto
   \left([\frac{2(x^2-4yx+2x-y^2+1)y}{(1+x-y)(1+x+y)(1-x-y)}:1],
   \frac{-4xy}{(1+x+y)(1-x-y)}\right),
\end{array}
$$
conjugates $\sigma_4$ to the automorphism $\tau\colon (x,y)\mapsto (-y,x)$
of $\A^2$. This automorphism~$\tau$ has eigenvalues $\pm \mathbf{i}$ over~$\C$.
If $v_1,v_2$ are eigenvectors $v_1,v_2$ then the invariant
subring is generated, over $\C$, by $v_1v_2$, $(v_1)^4, (v_2)^4$.
This implies that
$$
\begin{array}{rcl}
\Q[x,y]^{\tau}
&\!\!=\!\!& \Q[x^2+y^2,x^4+y^4,xy(x^2-y^2)]\\
&\!\!=\!\!& \Q[x_1,x_2,x_3]/(x_1^4+2x_3^2+2x_2^2-3x_1^2x_2).
\end{array}
$$
The surface $\M_2(4)/\langle \sigma_4\rangle$ is thus birational, over $\Q$,
to the hypersurface of $\A^3$ given by $x_1^4+2x_3^2+2x_2^2-3x_1^2x_2=0$.
This hypersurface is a conic bundle over the $x_1$ line,
with a section $(x_2,x_3) = (0,0)$, and is therefore rational.
An explicit birational map to $\A^2$ is
$(x_1,x_2,x_3)\dashmapsto (x_1,\frac{x_1x_3}{x_1^2-x_2})$,
whose inverse is
$(x,y) \dashmapsto
\left(x, \frac{(x^2+2y^2)x^2}{2(x^2+y^2)}, \frac{x^3y}{2(x^2+y^2)}\right)$.

$(iii)$ We embed $\M_2(3)$ in $\End_2$ as in \S\ref{MnP5Ak}.
It is given by maps $f$ which satisfy $f([0:1])=[1:0],\ f([1:0])=[1:1],\
f([1:1])=[1:0] $, and is thus parametrised by an open subset of
$\p^2$. A point $[a_1:a_3:a_4]$ corresponds to an endomorphism
$$
[u:v]\mapsto [a_3u^2+a_1uv+(-a_1-a_3)v^2, a_3u^2+a_4uv].
$$
Because the corresponding resultant is $a_3(a_3+a_4)(a_3+a_1)(a_1+a_3-a_4)$,
the surface $\M_2(3)$ is the complement of three lines in $\A^2$.

$(iv)$ The map $\sigma_3$ sends $(f,[0:1],[1:0],[1:1])$ to
$(f,[1:0],[1:1],[0:1])$, which is in the orbit of
$$
(gfg^{-1},g([1:0]),g([1:1]),g([0:1]))=(gfg^{-1},[0:1],[1:0],[1:1]),
$$
where $g\colon [u:v]\mapsto [u-v:u]$.
Because the endomorphism $gfg^{-1}$ equals
$$
[u:v]\mapsto [(a_1+a_3)u^2-(a_1+2a_3+a_4)uv+(a_3+a_4)v^2,
(a_1+a_3)u^2-(a_1+2a_3)uv],
$$
the automorphism $\sigma_3$ corresponds to the automorphism
$$
[a_1:a_3:a_4]\mapsto [-a_1-2a_3-a_4:a_1+a_3:-a_1-2a_3]
$$
of $\p^2$.  The affine plane where $a_4-a_1-a_3\neq 0$ is invariant,
and the action, in coordinates
$x_1=\frac{a_1}{a_4-a_1-a_3}$, $x_2=\frac{-a_1-2a_3-a_4}{a_4-a_1-a_3}$,
corresponds to $(x_1,x_2)\mapsto (x_2,-x_1-x_2)$. The quotient of
$\A^2$ by this action is rational over $\Q$ (see the proof of
Lemma~\ref{Lem:M26} below, where the quotient of the same action on~$\A^2$
is computed), so $\M_2(3)/\langle \sigma_3\rangle$ is rational over~$\Q$.
\end{proof}
\begin{lemma}\label{LemM1M2}
The varieties $\M_2(1)$, $\M_2(2)$ and $\M_2(2)/\langle \sigma_2\rangle$
are surfaces that are rational over~$\Q$.
\end{lemma}
\begin{proof}
 $(i)$ Let us denote by $U\subset \widetilde{\M}_2(1)$ the open
subset of pairs $(f,p)$ where $p$ is not a criticial point, which
means here that $f^{-1}(p)$ consists of two distinct points, namely
$p$ and another one. This open set is dense (its complement has
codimension~$1$) and is invariant by $\PGL_2$.
We consider the morphism
$$\begin{array}{ccl} \tau\colon\A^1\setminus \{0\} \times \A^1
\times \A^1\setminus \{0\}&\to& U\\
 (a,b,c) & \mapsto& ([u:v]\mapsto [uv:au^2+buv+cv^2], [0:1]),\end{array}$$
and observe that $\tau$ is a closed embedding. Moreover, the multiplicative
group $\mathbb{G}_m$ acts on $\A^1\setminus \{0\} \times \A^1 \times
\A^1\setminus \{0\}$ via $(t,(a,b,c))\mapsto (t^2a,tb,c)$, and
the orbits of this action correspond to the restriction of the
orbits of the action of $\PGL_2$ on $U$.

In consequence $U/\PGL_2$ is isomorphic to $(\A^1\setminus \{0\}
\times \A^1 \times \A^1\setminus \{0\})/\mathbb{G}_m$, which is isomorphic to
$\mathrm{Spec}{\mathbb{Q}}[\frac{b^2}{a},c,\frac{1}{c}]
=\A^1\times\A^1\setminus \{0\} $.
Hence, $U/\PGL_2$ is rational, and thus $\M_2(1)$ is rational as well.

$(ii)$ We take coordinates $[a:b:c:d]$ on $\p^3$ and consider the
open subset $W\subset \p^3$ where $ad\neq bc$, $b\neq 0$, and $c\neq0$.
The morphism
$$
\begin{array}{rccl}
W&\to& \widetilde{\M}_2(2)\\
\; [a:b:c:d] & \mapsto& ([u:v]\mapsto [v(au+bv):u(cu+dv)], [0:1],[1:0])\end{array}
$$
is a closed embedding.  Moreover, the multiplicative group $\mathbb{G}_m$
acts on $W$ via $(t,[a:b:c:d])\mapsto [a\mu^2: b\mu^3:c:\mu d]$,
and the orbits of this action correspond to the restriction of
the orbits of the action of $\PGL_2$ on $\widetilde{\M}_2(2)$.

In consequence $\widetilde{\M}_2(2)/\PGL_2$ is isomorphic to $W/\mathbb{G}_m$.
Denote by $\widehat{W}\subset \p^3$ the open subset where $bc\neq 0$,
which is equal to
$\mathrm{Spec}(
  \Q[\frac{a}{b},\frac{a}{c},\frac{b}{c},\frac{c}{b},\frac{d}{b},\frac{d}{c}]
)$.
Writing $t_1=\frac{d}{c}, t_2=\frac{a}{c}, t_3=\frac{b}{c}$, we also have
$\widehat{W}=\mathrm{Spec}(\Q[t_1,t_2,t_3,\frac{1}{t_3}])$,
and the action of $\mathbb{G}_m$ corresponds to $t_i\mapsto \mu^it_i$.
This implies that
$\widehat{W}/\mathbb{G}_m =\mathrm{Spec}(
  \Q[\frac{(t_1)^3}{t_3},\frac{t_1t_2}{t_3},\frac{(t_2)^3}{(t_3)^2}]
)$,
and is thus isomorphic to the singular rational affine hypersurface
 $\Gamma\subset \A^3$ given by $xz=y^3$. The variety $W/\mathbb{G}_m\cong
\widetilde{\M}_2(2)/\PGL_2$ corresponds to the open subset  of
$\Gamma$  where $y\neq 1$.

$(iii)$ The map $\sigma_2$ corresponds to the automorphism
$[a:b:c:d]\mapsto [d:c:b:a]$ of $\p^3$, to the automorphism
$(t_1,t_2,t_3)\mapsto (\frac{t_2}{t_3},\frac{t_1}{t_3},\frac{1}{t_3})$
of~$\widehat{W}$, and to the automorphism $(x,y,z)\mapsto (z,y,x)$ of $\Gamma$.
The invariant subalgebra of $\Q[x,y,z]/(xz-y^2)$ is thus generated
by $x+z,xz,y$. Hence the quotient of $\Gamma$ by the involution
is the rational variety $\mathrm{Spec}(\Q[x+z,y])=\A^2$.
This shows that $\M_2(2)/\langle \sigma_2\rangle$ is rational over~$\Q$.
\end{proof}
\begin{remark}
The proof of Lemma~\ref{LemM1M2} shows that $\M_2(2)$ is an affine
singular surface but that $\M_2(2)/\langle \sigma_2\rangle$ is
smooth. One can also see that the surface $\M_2(1)$ is singular,
and that it is not affine.
\end{remark}

\begin{lemma}\label{Lemm:M25quotient}
The surface $\M_2(5)/\langle \sigma_5\rangle$ is rational over~$\Q$.
\end{lemma}
\begin{proof}
The surface $\M_2(5)$ is isomorphic to an open subset of $\A^2$,
and an element $(x,y)\in \M_2(5)\subset \A^2$ corresponds to a
map $(f,([0:1], [1:0], [1:1], [x:1],[y:1]))\in \widetilde{\M}_2(5)$.
The element  $(f,([1:0], [1:1], [x:1],[y:1],[0:1]))$ is in the orbit under
$\Aut(\p^1)=\PGL_2$ of $(gfg^{-1},[0:1], [1:0], [1:1], [x-1:y-1],[1-x:1])$,
where $g\colon [u:v]\mapsto [(x-1)v:u-v]$.
Therefore the automorphism $\sigma_5$ of $\M_2(5)$ is the restriction of the
birational map $(x,y)\dasharrow \left(\frac{x-1}{y-1},1-x\right)$ of~$\A^2$,
which is the restriction of the following birational map of~$\p^2$
(viewing $\A^2$ as an open subset of~$\p^2$ via $(x,y)\mapsto [x:y:1]$):
$$
\tau\colon [x:y:z]\dasharrow \left[(x-z)z:(z-x)(y-z):(y-z)z\right].
$$
The map $\tau$ has order $5$ and the set of base-points of the
powers of $\tau$ are the four points $p_1=[1:0:0]$, $p_2=[0:1:0]$,
$p_3=[0:0:1]$, $p_4=[1:1:1]$. Denoting by $\pi\colon S\to \p^2$
the blow-up of these four points, the map $\hat{\tau}=\pi^{-1}Ê\tau\pi$
is an automorphism of the surface $S$.  Because $p_1,p_2,p_3,p_4$ are
in general position (no $3$ are collinear), the surface $S$ is a del~Pezzo
surface of degree~$5$, and thus the anti-canonical morphism
gives an embedding $S \to \p^5$ as a surface of degree~$5$.
The map $\pi^{-1}\colon \p^2\dasharrow S\subset\p^5$ corresponds
to the system of cubics through the four points, so we can assume,
up to automorphism of $\p^5$, that it is equal to
$$
\begin{array}{rcl}
\pi^{-1}([x:y:z])&=&[-xz(y-z): y(x-z)(x-y): z(x^2-yz):\\
&& (2yz-y^2-xz)z: (y-z)(yz+xy-xz): x(z-y)(y-z+x)],
\end{array}
$$
and the choice made here implies that $\hat{\tau}\in \Aut(S)$ is given by
$$
[X_0:\dots:X_5]\mapsto [X_0:X_1:X_3:X_4:X_5:-X_2-X_3-X_4-X_5].
$$
The affine open subset of $\p^5$ where $x_0\neq 0$ has the coordinates
$x_1=\frac{X_1}{X_0},x_2=\frac{X_2}{X_0},\dots,x_5=\frac{X_5}{X_0}$,
and is invariant.  The action of $\hat{\tau}$ on these coordinates is linear,
with eigenvalues $1,\zeta,\dots,\zeta^4$ where $\zeta$ is a $5$-th of unity.
We diagonalise the action over $\Q[\zeta]$,
obtaining eigenvectors $\mu_0,\mu_1,\dots,\mu_4$.
Then the field of invariant functions is generated by
$$
\mu_0, \mu_1\mu_4,\mu_2\mu_3,
(\mu_1)^2\mu_3,\mu_1(\mu_2)^2,(\mu_3)^2\mu_4,\mu_2(\mu_4)^2.
$$
In consequence, the field $\Q(S)^{\hat\tau}$ is generated by $x_1$
and by all invariant homogenous polynomials of degree $2$ and $3$
in $x_2,\dots,x_5$. The invariant homogeneous polynomials of degree
$2$ in $x_2,\dots,x_5$ are linear combinations of $v_1$ and $v_2$, where
$$
\begin{array}{rcl}
v_1&=&x_5^2+x_3x_5-x_3x_4+2x_2x_5+x_2x_4+x_2^2\\
v_2&=&x_4x_5+x_4^2+2x_3x_4+x_3^2-x_2x_5+x_2x_3.
\end{array}
$$
By replacing the $x_i$ by the composition with $\tau^{-1}$ given
above, we observe that $-1-11x_1+x_1^2-v_1-4v_2$ is equal to zero on~$S$,
so we can eliminate $v_1$.

The space of invariant homogeneous polynomials of degree $3$ in $x_2,\dots,x_5$
has dimension~$4$, but by again replacing the $x_i$ by their composition
with $\tau^{-1}$ we can compute that  the following invariant suffices:
$$
\begin{array}{rcl}
v_3 &=&
x_4x_5^2+x_4^2x_5+x_2x_5^2+2x_2x_4x_5+2x_2x_3x_5+x_2x_3^2+x_2^2x_5+x_2^2x_3.
\end{array}
$$
This shows that the field of invariants $\Q(S)^{\hat{\tau}}$ is
generated by $x_1,v_2,v_3$. In consequence, the map $S\dasharrow \A^3$
given by $(x_1,v_2,v_3)$ factors through a birational map from
$S/\hat{\tau}$ to an hypersurface $S'\subset\A^3$, and it suffices
to prove that this latter is rational. To get the equation of the
hypersurface, we observe that
$$
\begin{array}{ll}
3+10v_2+11v_2^2+4v_2^3+70x_1+445x_1^2+410x_1^3-85x_1^4+4x_1^5\\
+66x_1v_2^2-x_1^2v_2^2-30x_1^3v_2+320x_1^2v_2+140x_1v_2+v_3^2
\end{array}
$$
is equal to zero on $S$. Because the above polynomial is irreducible,
it is the equation of the surface $S'$ in $\A^3$. It is not clear
from the equation that surface is rational, so we will change coordinates.
Choosing $\mu= \frac{x_1^2-11x_1-1}{v_2}$ and $\nu=\frac{v_3}{v_2}$,
we have $\Q(S)^{\hat\tau}=\Q(x_1,\mu,\nu)$, and replacing
$v_2 =\frac{x_1^2-11x_1-1}{\mu}$ and $v_3=\nu\cdot\frac{x_1^2-11x_1-1}{\mu}$
in the equation above, we find  a simpler equation, which is

$
4x_1^2-\mu x_1^2+66\mu x_1+4 x_1 \mu^3-44x_1-30\mu^2 x_1
-4+3\mu^3+11\mu+\nu^2\mu-10\mu^2=0.
$

We do another change of coordinates, which is
$\kappa = \frac{nu}{x_1+7-5\mu}$, $\rho = \frac{\mu^2-5\mu+5}{x_1+7-5\mu}$,
and replace
$\nu = \frac{\kappa(\mu^2-5\mu+5)}{\rho}$,
$x_1 =\frac{-7\rho+5\rho\mu+\mu^2-5\mu+5}{\rho}$
in the equation, to obtain
$$
4\rho\mu-\mu+\kappa^2\mu+4+20\rho^2-20\rho=0
$$
which is obviously the equation of a rational surface. Moreover,
this procedure gives us two generators of $\Q(S)^{\hat{\tau}}$,
which are $\rho$ and $\kappa$.
\end{proof}
\begin{remark}The proof of Lemma~\ref{Lemm:M25quotient} could also
be seen more geometrically, using more sophisticated arguments.
The quotient of $(\p^1)^5$ by $\PGL_2$ is the del Pezzo surface
$S\subset \p^5$ of degree $5$ constructed in the proof. The action
of $\sigma_5$ on $S$ has two fixed points over $\C$, which are
conjugate over $\Q$. The quotient $S/\sigma_5$ has thus two singular
points of type $A_4$, and is then a weak del Pezzo surface of degree
$1$. Its equation in a weighted projective space $\p(1,1,2,3)$
is in fact the homogenisation of the one given in the proof of
the Lemma. The fact that it is rational can be computed explicitly,
as done in the proof, but can also be viewed by the fact that the
elliptic fibration given by the anti-canonical divisor has $5$
rational sections of self-intersection $-1$, and the contraction
of these yields another del Pezzo surface of degree $5$.
Moreover, any unirational del Pezzo surface of degree $5$
contains a rational point, and is then rational~\cite[page 642]{Isk}.
\end{remark}

\section{The surface $\M_2(6)$}
\subsection{Explicit embedding of the surfaces $\M_2(6)$ into $S_6$}

Using Lemma~\ref{Lemm:ProjectionAk}, one can see $\M_2(6)$ as a
locally closed surface in $\A^3$. However, this surface has an equation which
is not very nice, and its closure in $\p^3$ has bad singularities
(in particular a whole line is singular). Moreover, the action
of $\sigma_6$ on $\M_2(6)$ is not linear. We thus take another
model of $\M_2(6)$ (see the introduction for more details of how this
model was found), and view it as an open subset of the projective
hypersurface $S_6$ of $\p^3$ given by
$$
W^2 F_3(X,Y,Z) = F_5(X,Y,Z),
$$
where
$$
\begin{array}{rcl}
F_3(X,Y,Z)&=&(X+Y+Z)^3+(X^2Z+XY^2+Y\!Z^2)+2XY\!Z,\\
F_5(X,Y,Z)&=&(Z^3X^2+X^3Y^2+Y^3Z^2)-XY\!Z(Y\!Z+XY+XZ).
\end{array}$$

The following result shows that it is a projectivisation of $\M_2(6)$ that
has better properties, and directly shows Proposition~\ref{Prop:IntroDetailsM6}:
\begin{lemma}\label{Lemm:M6S6}
Let $\varphi\colon \A^3\dasharrow \p^3$ be the birational map given
by
 $$\varphi((x,y,z))=[-y+z-yz+xy: -y-z+yz+2x-xy: y-z-yz+xy: y+z-xy-yz],$$
$$\varphi^{-1}([W:X:Y:Z])=\left(\frac{(X+Z)(W+Y)}{W^2+XY+YZ+XZ},
\frac{W+Y}{X+Y}, \frac{(W+Z)(W+Y)}{W^2+XY+YZ+XZ}\right).$$
Then, the following hold:

$(i)$ The map $\varphi$ restricts to an isomorphism from $\M_2(6)\subset
\A^3$ to the open subset of $S_6$ which is the complement of the
union of the $9$ lines
$$
\begin{array}{ccc}
\begin{array}{l}
L_1:W=Z=-Y,\\
L_2:W=Y=-Z,\end{array}&
\begin{array}{l}
L_3: W=Y=-X,\\
L_4: W=X=-Y,\end{array}&
\begin{array}{l}
L_5: W=X=-Z,\\
L_6: W=Z=-X,\end{array}\\
L_7:-X=Y=Z,&
L_8:X=-Y=Z,&
L_9:X=Y=-Z,
\end{array}
$$
and of the $14$ conics
\begin{align*}
&\Cal_1:\left\{\begin{array}{l}W=X+Y+Z\\
  X^2+Y^2+Z^2+3(XY+XZ+YZ)=0\end{array}\right.;\\
&\Cal_2:\left\{\begin{array}{l}W=-(X+Y+Z)\\
  X^2+Y^2+Z^2+3(XY+XZ+YZ)=0\end{array}\right.;
\end{align*}

\vspace{-3mm}
\begin{align*}
&\Cal_3:\left\{
\begin{array}{l}W=X\\ X^2+XY+3XZ-YZ=0\end{array}\right.;&
&\Cal_4:\left\{
\begin{array}{l}W=-X\\ X^2+XY+3XZ-YZ=0\end{array}\right.;\\
&\Cal_5:\left\{
\begin{array}{l}W=Y\\ Y^2+YZ+3YX-ZX=0\end{array}\right.;&
&\Cal_6:\left\{
\begin{array}{l}W=-Y\\ Y^2+YZ+3YX-ZX=0\end{array}\right.;\\
&\Cal_{7}:\left\{
\begin{array}{l}W=Z\\ Z^2+ZX+3ZY-XY=0\end{array}\right.;&
&\Cal_{8}:\left\{
\begin{array}{l}W=-Z\\ Z^2+ZX+3ZY-XY=0\end{array}\right.;\\
&\Cal_{9}:\left\{
\begin{array}{l}Z=X\\ W(Y+3X)+X(X-Y)=0\end{array}\right.;&
&\Cal_{10}:\left\{
\begin{array}{l}Z=X\\ W(Y+3X)-X(X-Y)=0\end{array}\right.;\\
&\Cal_{11}:\left\{
\begin{array}{l}Y=Z\\ W(X+3Z)+Z(Z-X)=0\end{array}\right.;&
&\Cal_{12}:\left\{
\begin{array}{l}Y=Z\\ W(X+3Z)-Z(Z-X)=0\end{array}\right.;\\
&\Cal_{13}:\left\{
\begin{array}{l}
X=Y\\
 W(Z+3Y)+Y(Y-Z)=0
\end{array}\right.; &
&\Cal_{14}:\left\{
\begin{array}{l}X=Y\\ W(Z+3Y)-Y(Y-Z)=0\end{array}\right..
\end{align*}
Moreover, the union of the $23$ curves is
% the support on of the trace on $S_6$ of
the support of the zero-locus on $S_6$ of
$$
(W^2+XY+YZ+XZ)(W^2-X^2)(W^2-Y^2)(W^2-Z^2)(X^2-Y^2)(Y^2-Z^2)(Y^2-Z^2)
% =0
,
$$
which corresponds to the set of points where two coordinates are
equal up to sign, or where $W^2+XY+YZ+XZ=0$.

$(ii)$ The automorphism $\sigma_6$ of $\M_2(6)$ is the restriction
of the automorphism $$[W:X:Y:Z]\mapsto [-W:Y:Z:X]$$ of $\p^3$.

$(iii)$ To any point $[W:X:Y:Z]\in \M_2(6)\subset \p^3$ corresponds
the element $[a_0:\dots:a_5]\in \End_2\subset \p^5$ given by
$$
\begin{array}{rcl}
a_0&=&1,\\
a_1&=&\frac{(W-X)(W(X+Y+2Z)+Z(Y-X))}{(W^2+XY+XZ+YZ)(X-Z)}-1,\\
a_2&=& \frac{(W+Y)(W+Z)(X-W)(W(X+Y+2Z)+XY-Z^2)}{(W^2+XY+XZ+YZ)^2(X-Z)},\\
a_3&=& 1,\\
a_4&=& \frac{(Y+Z)(W-Z)(X-W)(W(2X+Y+Z)+YZ-X^2)}{((W^2+XY+XZ+YZ)(X+Z)(Y+W)(X-Z))}-1,\\
a_5&=&  0,
\end{array}$$
and its orbit of $6$ points is given by \begin{center}$[0:1]$,
$[1:0]$, $[1:1]$, $[{(X+Z)(W+Y)}:{W^2+XY+YZ+XZ}],[{W+Y}:{X+Y}]$,
$[(W+Z)(W+Y):{W^2+XY+YZ+XZ}]$. \end{center}
\end{lemma}
\begin{remark}\label{Rem:ChangeOrder}
Applying the automorphism of $\mathbb{P}^1$ given by
$$
[u:v]\mapsto [(W+X)u-(W+Y)v:(W-X)u],
$$ we can send the $6$ points of the orbits to
$$
[1:0], [W+Y:W-X], [0:1], [Z-W:X+Z], [1:1], [Y+Z,W+Z]
$$
respectively.
This also changes the coefficients of the endomorphism of $\mathbb{P}^1$.
\end{remark}
\begin{proof}
The explicit description of $\varphi$ and $\varphi^{-1}$ implies
that $\varphi$ restricts from an isomorphism $U\to V$, where $U\subset
\A^3$ is the open set where $y(y-1)(x-z)\neq 0$
and $V\subset \p^3$ is the open set where $(W^2+XY+YZ+XZ)(X+Y)(W+Y)(W-X)\neq 0$
(just compute $\varphi\circ \varphi^{-1}$ and $\varphi^{-1}\circ\varphi$).

Since $\M_2(6)\subset \A^3$ is contained in $U$,
the map $\varphi$ restricts to an isomorphism from
$\M_2(6)$ to its image, contained in~$V$.
Substituting
  $x_1=\frac{(X+Z)(W+Y)}{W^2+XY+YZ+XZ}$,
  $x_2=\frac{W+Y}{X+Y}$,
  $x_3=\frac{(W+Z)(W+Y)}{W^2+XY+YZ+XZ}$
into the formula of Lemma~\ref{Lemm:ProjectionAk} yields assertion~$(iii)$.
The fact that the map $f\in \End_2$ constructed by this process
sends $[x_1:1]$ to $[x_2:1]$ corresponds to the equation of the surface~$S_6$.
This shows that $\M_2(6)$ can be viewed, via $\varphi$, as an open subset
of~$S_6$.

The automorphism $\sigma_6$ sends a point $(x,y,z)\in \M_2(6)\subset \A^3$,
corresponding to an element  $(f, [0:1], [1:0], [1:1], [x:1],[y:1], [z:1])$
(see $\S\ref{MnP5Ak}$), to the point corresponding to
$\alpha=(f,  [1:0], [1:1], [x:1],[y:1], [z:1],[0:1])$.
The automorphism of $\p^1$ given by $\nu\colon [u:v]\mapsto [v(x-1):u-v]$
sends $\alpha$ to
$$
(\nu f\nu^{-1}, [0:1], [1:0], [1:1], [{x-1}:{y-1}],[{x-1}:{z-1}], [1-x:1]),
$$
so the action of $\sigma_6$ on $\M_2(6)\subset \A^3$ is the
restriction of the birational map of order $6$ given by
$$
\tau\colon (x,y,z)\dasharrow\left(\frac{x-1}{y-1},\frac{x-1}{z-1},1-x\right).
$$

Assertion $(ii)$ is then proved by observing that
$$\varphi^{-1}\tau\varphi([W:X:Y:Z])=[-W:Y:Z:X].$$

In order to prove $(i)$, we need to show that the complement of
$\M_2(6)\subset S_6$ is the union of $L_1,\dots,L_9$ and
$\Cal_1,\dots,\Cal_{14}$,
and that it is the set of points given by $W^2+XY+YZ+XZ=0$ or where
two coordinates are equal up to sign. Note that this complement
is invariant under $[W:X:Y:Z] \mapsto [-W:Y:Z:X]$,
since this automorphism corresponds to $\sigma_6\in \Aut(\M_2(6))$.
This simplifies the calculations.

 \smallskip

 Let us show that the union of $L_1,\dots,L_8,\Cal_1,\Cal_2,\Cal_9,\Cal_{10}$
is the zero locus of the polynomial $(W^2+XY+YZ+XZ)(W-X)(W+Y)(X+Y)(X-Z)$ :

$1)$ The zero locus of $W^2+XY+YZ+XZ$ on the quintic
gives the degree-$10$ curve
$$
\left\{
\begin{array}{l}
W^2+XY+YZ+XZ=0\\
(X+Z) (X+Y)(Y+Z) \left(X^2+Y^2+Z^2+3 X Y+3 X Z+3 Y Z\right)=0
\end{array}
\right.,
$$
which is the union of $L_1, L_2, \ldots, L_6$ and  $\Cal_1$ and $\Cal_2$
because the linear system of quadrics given by
$$
\begin{array}{l}
\lambda(W^2+XY+YZ+XZ) + \mu(X^2+Y^2+Z^2+3(XY+XZ+YZ)) = 0
\end{array}
$$
with $(\lambda:\mu)\in \p^1$ corresponds to
$$
\begin{array}{l}
\lambda(W^2-(X+Y+Z)^2) + (\lambda+\mu)(X^2+Y^2+Z^2+3(XY+XZ+YZ))=0,
\end{array}
$$
and thus its base-locus is the union of $\Cal_1$ and $\Cal_2$.

$2)$
Substituting $W=X$ in the equation of $S_6$ yields
$(X+Z)(Y+X)^2(X^2+XY+3XZ-YZ)=0$, so the locus of $W=X$\/ on $S_6$
is the union of $L_4$, $L_5$ and $\Cal_3$.

$3)$ Substituting $W=-Y$ in the equation yields $(X + Y) (Y + Z)^2
(3 X Y + Y^2 - X Z + Y Z)=0$ which corresponds to $L_1$, $L_4$ and $\Cal_6$.

$4)$ Substituting $X=-Y$ yields $(-W + Y) (W + Y) (Y - Z) (Y + Z)^2=0$
which corresponds to  $L_3$, $L_4$, $L_7$ and  $L_8$.

$5)$ Substituting $Z=X$ yields $(X+Y)(W(Y+3X)+X(X-Y))(W(Y+3X)-X(X-Y))$,
which corresponds to $L_8,\Cal_9,\Cal_{10}$.

Applying the automorphism $[W:X:Y:Z]\mapsto [-W:Y:Z:X]$, we obtain
that the union of $L_1,\dots,L_9,\Cal_1,\dots,\Cal_{14}$ is given
by the zero locusof $W^2+XY+YZ+XZ$ and all hyperplanes of the form
$x_i\pm x_j$ where $x_i,x_j\in \{W,X,Y,Z\}$ are distinct. This
shows that $(i)$ implies $(ii)$.

The steps $(1)-(4)$ imply that the complement of $S_6\cap V$ in
$S_6$ is the union of the lines $L_1,\dots,L_8$ and the conics
$\Cal_1,\Cal_2,\Cal_3,\Cal_6$.

On the affine surface $S_6\cap V$, the map $\varphi^{-1}$ is an
isomorphism. To any point $[W:X:Y:Z]\in S_6\cap V$ we associate
an element of $\p^5$, via the formula described in $(iii)$, which
should correspond to an endomorphism of degree $2$ if the point
belongs to $\M_2(6)$. Computing the resultant with the formula
of $\S\ref{SubSec:End2}$, we get a polynomial $R$  with many factors:
$$
\begin{array}{c}
R=(W^2+XY+XZ+XY)^4 (Z-X)(X+Z)^2(Y+Z)\\(W-Z)(W+Z)(Y+W)^3(W-X)^3
\\
 (W(X+Y+2Z)+XY-Z^2)(W(X+Y+2Z)+Z(Y-X))\\
 (W(2X+Y+Z)+YZ-X^2)(W(2X+Y+Z)+X(Y-Z)).
 \end{array}
$$
The surface $\M_2(6)$ is thus the complement in $S_6$ of the curves
$L_1$, $\dots$, $L_8$,  $\Cal_1$, $\Cal_2$, $\Cal_3$, $\Cal_6$,
and the curves given by the  polynomial $R$. The components
% $(W-X)(W+Y)(W^2+XY+XZ+YZ)(Z-X)$ were treated before.
$(W-X)(W+Y)$ $(W^2+XY+XZ+YZ)(Z-X)$ were treated before.
In particular, this shows (using again the action given by $\sigma_6$)
that each of the curves $L_1,\dots,L_9,\Cal_1,\dots,\Cal_{14}$
is contained in $S_6\setminus \M_2(6)$.
It remains to see that the trace of any irreducible divisor of $R$\/
on~$S_6$ is contained in this union.  The case of $(W^2+XY+XZ+XY)$
and all factors of degree $1$ were done before, so it remains to study
the last four factors. Writing $\Gamma=W^2 F_3-F_5$, which is the polynomial
defining $S_6$, we obtain
$$
\begin{array}{rc}(W(X+Y+2Z)+XY-Z^2)(-W(X+Y+2Z)+XY-Z^2)(X+Y)+\Gamma\\
=(Y-Z)(Y+Z)(X-Z)(W^2+XY+YZ+XZ)\\
(W(X+Y+2Z)+Z(Y-X))(-W(X+Y+2Z)+Z(Y-X))(X+Y)+\Gamma\\
=(Y-Z)(Y+Z)(X-Z)(W-X)(W+X).
\end{array}
$$
The last two factors being in the image of these two factors by
$[W:X:Y:Z]\mapsto [W:Y:Z:X]$, we have shown that $S_6\setminus\M_2(6)$
is the union of $L_1,\dots,L_9,\Cal_1,\dots,\Cal_{14}$.
\end{proof}

\begin{corollary}\label{Coro:Generaltype}
The variety $\M_2(6)$ is an affine smooth surface, which is birational
to $S_6$, a projective surface of general type.
\end{corollary}
\begin{proof}
Computing the partial derivatives of the equation of $S_6$, one
directly sees that it has exactly $11$ singular points,
of which two are fixed by
$$
[W:X:Y:Z]\mapsto [-W:Y:Z:X]
$$
and the $9$ others form a set that consists of
an orbit of size $3$ and an orbit of size $6$:
$$[1: 0: 0: 0], [0: 1: 1: 1]$$
$$
\begin{array}{rrr}
[0: 1: 0: 0],& [0: 0: 1: 0],& [0: 0: 0: 1],\vspace{0.2 cm}\\
\ [-1: -1: 1: 1],& [-1: 1: -1: 1], &[-1: 1: 1: -1],\\
\ [1: -1: 1: 1],&  [1: 1: -1: 1],& [1: 1: 1: -1].
\end{array}
$$
Since none of the points belongs to $\M_2(6)$, viewed in $S_6$
using Lemma~\ref{Lemm:M6S6}, the surface $\M_2(6)$ is smooth.
Because the complement of $\M_2(6)$ in $S_6$ is the zero locus of
a homogeneous polynomial (by Lemma~\ref{Lemm:M6S6}(ii)),
the surface $\M_2(6)$ is affine.
It remains to see that $S_6$ is of general type.

The point $[1:0:0:0]$ is a triple point, and all others are double points.
Denoting by $\pi\colon \widehat{\p^3}\to \p^3$ the blow-up of the $11$ points,
the strict transform $\what{S_6}$ of $S_6$ is a smooth surface.

We denote by $E_1,\dots,E_{11}\in \mathrm{Pic}(\widehat{\p^3})$ the
exceptional divisors obtained (according to the order above),
and by~$H$\/ the pull-back of a general hyperplane of~$\p^3$.
The ramification formula gives the canonical divisor
$K_{\widehat{\p^3}} = -4H+2\sum_{i=1}^{11} E_i$.
The divisor of $\what{S_6}$ is then equivalent to
$5H - 3E_1 - 2\sum_{i=2}^{11} E_i$.
Applying the adjunction formula, we find that
$K_{\what{S_6}} = 
(K_{\widehat{\p^3}}+\what{S_6})|_{\what{S_6}}=(H-E_1)|_{\what{S_6}}$.

The linear system $H-E_1$ corresponds to the projection $\p^3\dasharrow\p^2$
given by $[W:X:Y:Z]\dashmapsto [X:Y:Z]$. The map
$K_{\what{S_6}}\stackrel{\lvert K_{\what{S_6}}\rvert}{\longrightarrow} \p^2$
is thus surjective, which implies that $\what{S_6}$, and thus $S_6$,
is of general type.
\end{proof}

\begin{remark}
% In Diophantine Geometry a natural question is if the set of $S$--integral
% points of a projective smooth variety is Zariski dense with respect
% to a divisor (e.g. see \cite{Voj} for the definitions). In this
% context it is important to know if the divisor \textquotedblleft
% at infinity\textquotedblright\ has normal crossing intersections.
% See for example Vojta's conjecture (\cite{Voj}). Hence it could
% be interessant to know if the divisor of $\hat{S_6}$ given by $\sum_{1\leq
% i\leq 9} \hat{L}_i+\sum_{1\leq i\leq 14} \hat{\Cal}_i$ is normal
% crossing; where the $\hat{L}_i$'s and the $\hat{\Cal}_i$'s are
% the lines and the conics in $\hat{S_6}$ associated to the  $L_i$'s
% and $\Cal_i$'s in $S_6$. With some calculations it is possible
% to see that the intersections of the $L_i$'s and $\Cal_i$'s in
% $S_6$ are not normal only in the following situations:

One can see that the divisor of $\what{S_6}$ given by
$D=\sum_{1\leq i\leq 9} \hat{L}_i+\sum_{1\leq i\leq 14} \what{\Cal}_i$
is normal crossing, where the $\hat{L}_i$'s and the $\what{\Cal}_i$'s are
the lines and the conics in $\what{S_6}$ associated to the $L_i$'s
and $\Cal_i$'s in $S_6$. This follows from the study of the intersections
of the $L_i$'s and $\Cal_i$'s in $S_6$. They are not normal only
in the following situations:

i) $\Cal_3\cap\Cal_9=\{[0:0:1:0]\}$ and the curves have a common
tangent line, which is $W=X=Z$. By applying the automorphism $\sigma_6$
we find that the intersection is not normal also in $\Cal_6\cap\Cal_{14},
\Cal_7\cap\Cal_{11}, \Cal_4\cap\Cal_{10},\Cal_5\cap\Cal_{13}$ and
$\Cal_8\cap\Cal_{12}$.  We note that in each of the previous cases
the intersection point is a singular point of $S_6$ and the multiplicity
of intersection is $2$ because the intersection point is between
two non-coplanar conics. Therefore in $\hat{S_6}$ the intersections
become normal.

ii)  $\Cal_9\cap\Cal_{10}=\{[0:0:1:0],[0:1:1:1], [1:0:0:0]\}$.
In this case the two conics are coplanar and the tangent point
is $[1:0:0:0]$ with tangent line $Z=X=-Y/3$. Note that all three
intersection points, in particular $[1:0:0:0]$, are singular
points of~$S_6$. Hence $\widehat{\Cal_9}$ and $\what{\Cal_{10}}$ have
normal crossings in $\what{S_6}$. By applying the automorphism $\sigma_6$
we find the similar situation with $\Cal_{11}\cap\Cal_{12}$ and
$\Cal_{13}\cap\Cal_{14}$.

Therefore the divisor $D$\/ of $\what{S_6}$ is normal crossing.
\end{remark}

\begin{remark}
The condition for the divisor at infinity to be normal crossing
is sometimes connected to the notion of integral points of surfaces
(see Section \ref{sip} for the definition of $S$--integral point).
For an example in this direction, see Vojta's conjecture~\cite{Voj}.
In our situation it will be easy to prove the finiteness
of integral points of $M_2(6)$ and we shall do this in Section~\ref{sip}.
But in general the study of the integral points on surfaces
could be a very difficult problem. See for example \cite{CZ} for
some results in this topic.
\end{remark}

\subsection{Quotients of $\M_2(6)$} It follows from the description of
$\M_2(6)\subset S_6$ given in Lemma~\ref{Lemm:M6S6} that the quotient
$(\M_2(6))/\langle \sigma_6^3\rangle$ is rational over~$\Q$:
the quotient map corresponds to the projection
$\M_2(6)\to \p^2$ given by ${[W:X:Y:Z]} \mapsto {[X:Y:Z]}$,
whose image is an affine open subset $\mathcal{U}$ of $\p^2$,
isomorphic to $(\M_2(6))/\langle \sigma_6^3\rangle$.
The ramification of $\M_2(6)\to (\M_2(6))/\langle \sigma_6^3\rangle$
is the zero locus of $F_5$ on the open subset~$\mathcal{U}$.
We now describe the other quotients:
\begin{lemma}\label{Lem:M26}
The quotient $(\M_2(6))/\langle\sigma_6^2\rangle$ is birational
to a projective surface of general type, but the quotient
$(\M_2(6))/\langle\sigma_6\rangle$ is rational.
\end{lemma}
\begin{proof}
$(i)$ Recall that $S_6\subset \p^3$  has equation $W^2F_3(X,Y,Z)=F_5(X,Y,Z)$
and $\sigma_6$ corresponds to $[W:X:Y:Z]\mapsto [-W:Y:Z:X]$
(Lemma~\ref{Lemm:M6S6}).
In particular, the projection $\p^3\dasharrow \p^2$ given by
$[W:X:Y:Z]\dashmapsto [X:Y:Z]$ corresponds to the quotient map
$U\to U/\langle\sigma_6^3\rangle$,
where $U\subset S_6$ is the open subset where $F_3\neq 0$.
This implies that the surfaces $S_6/\langle\sigma_6^3\rangle$
and $\M_2(6)/\langle\sigma_6^3\rangle$ are rational over $\Q$,
and that $\M_2(6)/\langle\sigma_6^2\rangle$ is birational to
the quotient of $\p^2$ by the cyclic group of order $3$ generated
by $\mu\colon [X:Y:Z]\mapsto [Y:Z:X]$.  We next prove that this
quotient of~$\p^2$ is rational.

The open subset of $\p^2$ where $X+Y+Z\neq 0$ is an affine plane
$\A^2$ invariant under $\mu$. We choose coordinates $x_1=\frac{X-Y}{X+Y+Z}$
and $x_2=\frac{Y-Z}{X+Y+Z}$ on this plane, and compute that the
action of $\mu$ on $\A^2$ corresponds to $(x_1,x_2)\mapsto (x_2,-x_1-x_2)$.
This action is linear, with eigenvalues $\omega,\omega^2$ where
$\omega$ is a third root of unity.  We diagonalise the action over
$\Q[\omega]$, obtaining eigenvectors $w_1,w_2$.
the invariant ring is generated by $w_1w_2$, $w_1^3$, and $w_2^3$.
In consequence, the ring $\Q[x_1,x_2]^{\mu}$ is generated by
the invariant homogeneous polynomials of degree $2$ and~$3$.
An easy computation gives the following generators
of the vector spaces of invariant polynomials of degree $2$ and~$3$:
$$
\begin{array}{rcl}
v_1 &=& x_2^2+x_1x_2+x_1^2,\\
v_2 &=& x_1x_2^2+x_1^2x_2,\\
v_3 &=& x_1^3-x_2^3-3x_1x_2^2.
\end{array}$$
Hence, $\Q[x_1,x_2]^{\mu}=\Q[v_1,v_2,v_3]$.
Since $v_1^3-9v_2^2-3v_2v_3-v_3^2=0$,
the quotient $\A^2/\langle \mu\rangle$ is equal to the affine singular
cubic hypersurface of $\A^3$ defined by the corresponding equation.
The projection from the origin gives a birational map from the
cubic surface to $\p^2$.  Hence $\A^2/\langle \mu\rangle$,
and thus $\M_2(6)/\langle \sigma_6\rangle$, is rational over $\Q$.

$(ii)$ We compute the quotient $\M_2(6)/\langle\sigma_6^2\rangle$.
The open subset of $\p^3$ where $W\neq 0$ is an affine space;
we choose coordinates $x_0=\frac{X+Y+Z}{W}$, $x_1=\frac{X-Y}{W}$ and
$x_2=\frac{Y-Z}{W}$, and the action of $\sigma_6^4$ (which is the inverse of
$\sigma_6^2$) corresponds to $(x_0,x_1,x_2)\mapsto (x_0,x_2,-x_1-x_2)$.
In these coordinates,
$\Q[x_0,x_1,x_2]^{\langle\sigma_6^2\rangle}=\Q[x_0,v_1,v_2,v_3]$,
where $v_1,v_2,v_3$ are polynomials of degree $2,3,3$ in $x_1,x_2$,
which are the same as above.

This implies that the quotient of $\A^3$ by $\sigma_6^2$ is the
rational singular threefold $V\subset \A^4=\mathrm{Spec}(\Q[x_0,v_1,v_2,v_3])$
with the equation $v_1^3-9v_2^2-3v_2v_3-v_3^2=0$. The threefold
$V$ is birational to $\p^3$ via the map $(x_0,v_1,v_2,v_3)\dashmapsto
[v_1:v_1x_0:v_2:v_3]$, whose inverse is
$[W:X:Y:Z]\dashmapsto (\frac{X}{W},
  \frac{9Y^2+3YZ+Z^2}{W^2}:\frac{Y(9Y^2+3YZ+Z^2)}{W^3},
  \frac{Z(9Y^2+3YZ+Z^2)}{W^3}
)$.
We write the equation for $S_6$ in $\A^3$ in terms of our invariants,
and find that its image in~$V$\/ is given by the zero-locus of
$$
x_0^3(32-2v_1)+3v_3x_0^2-6v_1x_0-12v_2+v_3-v_1(v_3-3v_2),
$$
which is this birational with $\M_2(6)/\langle\sigma_6^2\rangle$.
That zero-locus is in turn birational, via the map $V\dasharrow\p^3$
defined above, to the quintic hypersurface of $\p^3$ given by
$$
W^2\wtilde{F_3}(X,Y,Z)=\wtilde{F_5}(X,Y,Z),
$$
where
$$
\begin{array}{rcl}
\wtilde{F_3}(X,Y,Z) & \!\! = \!\! & (Z+2X)(16X^2-8XZ+Z^2-27Y^2-9Y\!Z)-108Y^3,\\
\wtilde{F_5}(X,Y,Z) & \!\! = \!\! &
  X^2(9Y^2+3Y\!Z+Z^2)(2X-3Z)-(3Y-Z)(9Y^2+3Y\!Z+Z^2)^2.
\end{array}
$$
We can see that $[1:0:0:0]$ is the only triple point of this surface,
and that all other singularities are ordinary double points.
As in the proof of Corollary~\ref{Coro:Generaltype}, this implies
that the surface is of general type.
\end{proof}

\subsection{Rational points of $\M_2(6)$}\label{rat}
Because $\M_2(6)$ is of general type, the Bombieri--Lang conjecture
asserts that the rational points of $\M_2(6)$ should not be Zariski dense.
Since $\M_2(6)$ has dimension~$2$, this means that
there should be a finite number of curves that contain all but
finitely many of the rational points; moreover, by Faltings' theorem
(proof of Mordell's conjecture, which is the \hbox{$1$-dimensional}
case of Bombieri--Lang), each curve that contains infinitely many
rational points has genus $0$ or~$1$.

% The projective surface $S_6$ contains many rational curves, at
% least the $23$ irreducible rational curves of $S_6\setminus \M_2(6)$,
% but we do not know if one rational curve meets $\M_2(6)$. However,
% we can prove that $\M_2(6)$ contains infinitely many rational points,
% by looking at elliptic curves:

We already found $23$ rational curves on $S_6$, namely the
components of $S_6\setminus \M_2(6)$.  We searched in several ways 
for other curves that would yield infinite families of rational points 
on $\M_2(6)$.  One approach was to intersect the quintic surface
$W^2 F_3 = F_5$ with planes and other low-degree surfaces that
contain some of the known rational curves, hoping that the
residual curve would have new components of low genus.
We found no new rational curves this way but did discover
several curves of genus $1$, some of which have infinitely many 
rational points.  Another direction was to pull back curves of 
low degree on the $[X:Y:Z]$ plane on which $F_3 F_5$ has several 
double zeros (where the curve is either tangent to $F_3 F_5=0$
or passes through a singular point).  This way we found a few of 
the previous elliptic curves, and later also a $24$-th rational curve.
We next describe these new curves of genus $0$ and $1$ on $\M_2(6)$.

\begin{lemma}\label{rational}
The equation
$$
X^3 + Y^3 + Z^3 = X^2 Y + Y^2 Z + Z^2 X
$$
defines a rational cubic in $\p^2$, birational with $\p^1$ via the map
$c: \p^1 \to \p^2$ taking $[m:1]$ to
$$
[-m^3+2m^2-3m+1: m^3-m+1: m^3-2m^2+m-1].
$$
It is smooth except for the node at $[X:Y:Z] = [1:1:1]$.
Its preimage under the $2:1$ map $\M_2(6) \to \p^2$ taking
$[W:X:Y:Z]$ to $[X:Y:Z]$ is a rational curve $C$
that is mapped to itself by $\sigma_6$.  Every point of~$C$\/
parametrises a quadratic endomorphism $f: \p^1 \to \p^1$
that has a rational fixed point in addition to its rational \hbox{$6$-cycle}.
\end{lemma}

\begin{proof} We check that the coordinates of $c$ are relatively prime
cubic polynomials, whence its image is a rational cubic curve, and that
these coordinates satisfy the cubic equation
$X^3 + Y^3 + Z^3 = X^2 Y + Y^2 Z + Z^2 \kern-.05ex X$.
This proves that $c$ is birational to its image,
and thus that our cubic is a rational curve.
A rational cubic curve in $\p^2$ has only one singularity,
and since ours has a node at $[1:1:1]$ there can be no other singularities.
(The parametrisation $c$ was obtained in the usual way by projecting from
this node.)  We may identify the function field of the cubic with $\Q(m)$.
Substituting the coordinates of $c$ into $F_3$ and $F_5$ yields
$-4(m^2-m)^3$ and $-4(m^2-m)^3(m^2-m+1)^3$ respectively.
Thus adjoining a square root of $F_5/F_3$ yields the quadratic
extension of $\Q(m)$ generated by a square root of $m^2-m+1$.
The conic $j^2 = m^2-m+1$ is rational (e.g.\ because it has
the rational point $(m,j)=(0,1)$), so this function field
is rational as well.  The cubic $X^3 + Y^3 + Z^3 - (X^2 Y + Y^2 Z + Z^2 X)$
is invariant under cyclic permutations of $X,Y,Z$\/ (which act on
the projective \hbox{$m$-line} by the projective linear transformations
taking $m$ to $1/(1-m)$ and $(m-1)/m$); hence $\sigma_6^2$ acts on~$C$.
Also $\sigma_6^3$ acts because it fixes $X,Y,Z$\/ and takes $W$\/ to $-W$,
which fixes $m$ and takes $j$ to $-j$ in $j^2 = m^2-m+1$.  This proves that
$\sigma_6$ takes $C$\/ to itself.

Choosing a parametrisation of the conic given by
$$(j,m) = \left(\frac{p^2+p+1}{1-p^2}, \frac{2p+1}{1-p^2}\right),
$$ and putting the first, third, and fifth points of the cycle at
$\infty$, $0$, and $1$ respectively (see Remark~\ref{Rem:ChangeOrder}),
we find that the generic quadratic endomorphism $f$\/ parametrised by~$C$\/
has the \hbox{$6$-cycle}
\begin{eqnarray*}
x_1 = [1:0], && x_2 = [p^3+5p^2+2p+1 : (2p+1)(p^3+p^2+1)], \\
x_3 = [0:1], && x_4 = [(p+2)(p^3-p^2-2p-1) : 2(p-1)(p+1)^2(2p+1)], \\
x_5 = [1:1], && x_6 = [2(p+2)(2p+1) : (p+1)(p^3+p^2+4p+3)].
\end{eqnarray*}
The quadratic map that realises the above cycle $x_1,\ldots,x_6$ is
$$
[u:v]\mapsto
  [(2p+1)(p^3+p^2+1)(u-\lambda_1 v)(u-\lambda_2v) :
  (p^5+5p^2+2p+1)(u-\lambda_3v)(u-\lambda_4 v)],
$$
where
$$
\begin{array}{rclrcl}
\vphantom{\Big)}\lambda_1&=&\frac{p^3+5p^2+2p+1}{(1+2p)(p^3+p^2+1)},&
\lambda_2&=&\frac{(p+2)^2(p^3-p^2-2p-1)^2}{(p^3+5p^2+2p+1)(p^3+p^2+4p+3)(p+1)^2},\\
\vphantom{\Big)}\lambda_3&=&\frac{2(p+2)(1+2p)}{(p+1)(p^3+p^2+4p+3)},&
\lambda_4&=&\frac{(p^2-1)(p^3+5p^2+2p+1)(p^3-p^2-2p-1)}{(2p+1)^2(p^3+p^2+1)^2}.
\end{array}
$$
%The machine-readable formula 
%f(x)=(2*p^4*x+3*p^3*x-p^3+p^2*x-5*p^2+2*p*x-2*p+x-1)*(p^8*x-p^8+8*p^7*x-2*p^7+24*p^6*x+7*p^6+54*p^5*x+18*p^5+87*p^4*x-2*p^4+84*p^3*x-36*p^3+47*p^2*x-41*p^2+16*p*x-20*p+3*x-4)/((p^4*x+2*p^3*x+5*p^2*x-4*p^2+7*p*x-10*p+3*x-4)*(4*p^8*x-p^8+12*p^7*x-4*p^7+13*p^6*x+6*p^6+14*p^5*x+16*p^5+17*p^4*x+5*p^4+10*p^3*x-8*p^3+6*p^2*x-9*p^2+4*p*x-4*p+x-1)*(1+p))
The point $x_0 = [1:p+1]$ is fixed by $f$.
[We refrain from exhibiting the coefficients of~$f$\/ itself,
which are polynomials of degree~$11$, $12$, and $13$ in~$p$.
However, we include a machine-readable formula for~$f$\/
in a comment line of the \LaTeX\ source
following the displayed formula for the $\lambda_i$,
so that the reader may copy the formula for~$f$\/ from the arXiv preprint
and check the calculation or build on it.]
\end{proof}

\begin{lemma}\label{XYZ}
The sero locus on $\M_2(6)$ of $XY\!Z$ is the union of
three isomorphic elliptic curves, which contain infinitely many
rational points.
\end{lemma}
\begin{proof}
The 3-cycle $[W:X:Y:Z]\mapsto [W:Y:Z:X]$ permutes the factors
$X,Y,Z$\/ of $XY\!Z$, so it is enough to consider only curve defined by $Z=0$.
Substituting $Z=0$ in the equation of $S_6$, we obtain
$$
W^2(X^3+3X^2Y+4XY^2+Y^3)-Y^2X^3=0.
$$
This is a singular plane curve of degree $5$, birational to the
smooth elliptic curve $\rmE \subset \A^2$ given by
\beq
\label{el1}
y^2 = x^3 + 4x^2 + 3x + 1 = (x+1)^3 + x^2
\eeq
via the map $[W:X:Y]\dasharrow (\frac{Y}{X},\frac{Y}{W})$
whose inverse is $(x,y)\mapsto [\frac{1}{y}:\frac{1}{x}:1]$.

We show that $\rmE$ has positive \MW\ rank by showing that
the points $(0,\pm 1)\in \rmE$ have infinite order.
Using to the duplication formula in \cite[p.31]{ST} we calculate that
the $x$~coordinate of $2(x,y)$ is 
$$
\frac{x^4-6 x^2-8 x-7}{4 x^3+16 x^2+12 x+4}.
$$
Note that the $x$ coordinate of $2(0,1)$ is $-7/4$.  Hence
$(-7/4, 13/8)\in \rmE$. The Nagell-Lutz Theorem (e.g.\ see \cite[p.56]{ST})
implies that $(-7/4, 13/8)$ is not a torsion point. According to
Lemma~\ref{Lemm:M6S6}, the curve $\rmE$ meets the boundary
$S_6\setminus \M_2(6)$ in finitely many points; therefore
infinitely rational points of~$\rmE$ belong to $\M_2(6)$.
\end{proof}

\begin{remark}
In fact the curve $\rmE$ has a simple enough equation that we readily
determine the structure of its \MW\ group $\rmE(\Q)$ using
\hbox{$2$-descent} as implemented in Cremona's program \textsc{mwrank}:
$\rmE(\Q)$ is the direct sum of the 3-element torsion group generated by
$(-1,1)$ with the infinite cyclic group generated by $(0,1)$.
The curve $\rmE$ has conductor $124$, small enough that it already
appeared in Tingley's ``Antwerp'' tables \cite{MFIV} of curves
of conductor at most $200$, where $\rmE$ is named 124B (see page~97);
Cremona's label for the curve is 124A1 \cite[p.102]{Cremona}.
(Both sources give the standard minimal equation
$y^2 = x^3 + x^2 - 2x + 1$ for $\rmE$, obtained from~(\ref{el1})
by translating $x$ by~$1$.)

\end{remark}

Lemma \ref{XYZ} yields the existence of infinitely many classes
of endomorphisms of $\Proj^1$ defined over $\Q$ that admit a rational
periodic point of primitive period~$6$.  We next give an example
by using the previous arguments.
\begin{example}
The point $(-7/4, 13/8)$ is a point of the elliptic curve $\rmE$ defined
in~(\ref{el1}). Using the birational map in the proof of Lemma~\ref{XYZ}
we see that the point $(x,y)=(-7/4, 13/8)$ is sent to
the point $[W:X:Y:Z]=[8/13:-4/7:1:0]$. Applying the map $\phi^{-1}$ of
Lemma \ref{Lemm:M6S6}, we obtain the point $(91/19,49/13,-98/19)\in \A^3$.
Finally we apply Lemma \ref{Lemm:ProjectionAk}, finding
the endomorphism
$$
f[u:v] = [(19u+98v)(133u-441v) : 19u(133u-529v)],
$$
which admits the $6$-cycle
$$
[0:1]\mapsto[1:0]\mapsto[1:1]\mapsto[91:19]\mapsto[49:13]\mapsto[-98:19]\mapsto[0:1].
$$
\end{example}
Apart from the elliptic curves in $S_6$ corresponding to $XY\!Z=0$,
there  are other elliptic curves which can be found using the
special form of the equation of $S_6$. The following lemma shows
that none of these curves provides a rational point of $\M_2(6)$.

\begin{lemma}
The intersection with $S_6$ of the hyperplanes $W=\pm(X+Y+Z)$ is the union
of two conics, contained in $S_6\setminus \M_2(6)$, and two isomorphic
elliptic curves, which contain only finitely many rational points,
all contained in $S_6\setminus \M_2(6)$ too.
\end{lemma}
\begin{proof}
Thanks to the automorphism $[W:X:Y:Z]\mapsto [-W:X:Y:Z]$, it is
sufficient to study the curve defined by $W=X+Y+Z$. Replacing $W=X+Y+Z$
in the equation of $S_6$ yields the following reducible polynomial
of degree $5$:
$$
\left(X^2+Y^2+Z^2+3(XY+XZ+Y\!Z)\right)
\left((X+Y+Z)^3-X^2Y-Y^2Z-XZ^2)\right).
$$

The first factor corresponds to the conic $\Cal_1\subset S_6\setminus \M_2(6)$
(Lemma \ref{Lemm:M6S6}) and the second yields a smooth plane cubic
$E\subset S_6$ birational to the elliptic curve $\mathcal{E}\subset \p^2$
given in Weierstrass form by
$$
y^2z=4x^3+z^3
$$
via the birational transformation $\psi\in \Bir(\p^2)$ given by
$$
\psi\colon [X:Y:Z]\dashmapsto
 \left[(-X-Y)(Y+Z): -X^2-2YZ-(X+Y+Z)^2 :(Y+Z)^2\right],
$$
whose inverse is given by
$$
[x:y:z]\dashmapsto
 \left[-2 x^2-2 x z+y z+z^2:-2 x^2+2 x z-y z-z^2:{2 x^2+2x z+y z+z^2}\right].
$$

Note that the image of $\psi(E\setminus \M_2(6))$ is contained
in open set where $x z\neq 0$. The result will then follow from
the fact that $\mathcal{E}(\Q)=\{[0:1:0],[0:1:1],[0,-1,1]\}$,
which we prove next.  We could again do this using \textsc{mwrank},
or by finding the curve in tables (the conductor is $27$),
but here it turns out that the result is much older,
reducing to the case $n=3$ of Fermat's Last Theorem.

It is clear that the three points are in $\mathcal{E}$.
Conversely, let $[x:y:z]\in \mathcal{E}$ be a rational point.
Write the equation of $\mathcal{E}$ as $4x^3=z(y-z)(y+z)$,
and make the linear change of coordinates $(y,z)=(r+s,r-s)$.
This yields $z(y-z)(y+z) = 4rs(s-r)$ and $x^3 = rs(s-r)$,
and we are to show that $x=0$.
We may assume that $x,y,z$ are integers with no common factor.
Then $\gcd(r,s)=1$ (since any prime factor of both $r$ and $s$ 
would divide $x^3$ and would thus also be a factor of~$x$).
Therefore $r,s-r,s$ are pairwise coprime, and if their product
is a nonzero cube then each of them is a cube individually.
But then the cube roots, call them $\alpha,\beta,\gamma$, satisfy
$\alpha^3+\beta^3=\gamma^3$.  Hence by the $n=3$ case of Fermat
$\alpha\beta\gamma=0$, and we are done.
\end{proof}

There may be yet further rational curves to be found:
we searched for rational points on $\M_2(6)$ using the
\hbox{$p$-adic} variation of the technique of~\cite{NDE:ANTS4},
finding more than $100$ orbits under the action of $\langle\sigma_6\rangle$
that are not accounted for by the known rational and elliptic curves,
such as $[W:X:Y:Z] = [-46572 : 20403 : 35913 : 16685]$
and $[-75523 : 54607 : 72443 : 62257]$.

We conclude this section with a curiosity involving the curves
$F_3(X,Y,Z)=0$ and $F_5(X,Y,Z)=0$, which lift to points of $S_6$
fixed by $\sigma_6^3$.  These curves have genus $1$ and $2$ respectively,
and do not yield any rational points on $\M_2(6)$.  But they are
birational with the modular curves $X_1(14)$ and $X_1(18)$
which parametrise elliptic curves with a $14$- or \hbox{$18$-torsion}
point respectively.  Could one use the modular structure to explain
these curves' appearance on $S_6$?

\subsection{$S$--integral points of $M_2(6)$ }\label{sip}
In this section we consider the $S$--integral points of $M_2(6)$
viewed as $S_6\setminus D$ where $D$ is the effective ample
divisor $D=\sum_{1\leq i\leq 9} L_i+\sum_{1\leq i\leq 14} \Cal_i$
and the lines $L_i$ and the conics $\Cal_i$ are the one defined
in Lemma \ref{Lemm:M6S6}.
We shall apply the so called $S$-unit Equation Theorem, but before
to state it we have to set some notation.

Let $K$ be a number field and $S$ a finite set of places of $K$
containing all the archimedean ones. The set of $S$--integers is
the following one:
$$\mathcal{O}_S\coloneqq \{x\in K\mid |x|_v\leq 1\ \text{for any
$v\notin S$}\}$$
and we denote by $\mathcal{O}_S^*$ its group of units
$$\mathcal{O}_S^*\coloneqq \{x\in K\mid |x|_v= 1\ \text{for any
$v\notin S$}\}$$
which elements are called $S$--units. (See for example \cite{BoGu}
for more information about these objects.)

We shall use the following classical result:

\begin{theorem}\label{sunit} Let $K$, $S$ and $\mathcal{O}_S^*$
be as above. Let $a$ and $b$ be nonzero fixed elements of $K$.
Then the equation
$$ax+by=1$$ has only finitely many solutions $(x,y)\in (\mathcal{O}_S^*)^2$.
\end{theorem}

This result, due to Mahler, was proved in some less general form
also by Siegel. Theorem \ref{sunit} can be viewed as a particular
case of the result proved by Beukers and Schlickewei in \cite{BeSc}
that gives also a bound for the number of solutions.

We recall briefly the notion of $S$--integral points. Let $X\subset
\mathbb{A}^n$ be an affine variety defined over a number field
$K$ with $\mathcal{O}$ its ring of algebraic integers. Let $K[X]$
be the ring of regular functions on $X$. Recall that $K[X]$ is
a quotient of the polynomial ring $K[x_1,x_2,\ldots,x_n]$. Denote
by $\mathcal{O}_S[X]$ the image in this quotient of the ring
$\mathcal{O}_S[x_1,x_2,\ldots,x_n]$.  If $P=(p_1,p_2,\ldots, p_n)$
is a point of $X$ whose coordinates are all $S$--integers,
then $P$\/ defines a morphism of specialisation
$\phi_P\colon \mathcal{O}_S[X] \to \mathcal{O}_S$.
It is clear that also the converse holds: to any such morphism
corresponds a point $P\in X$ with $S$--integral coordinates.

Let $\tilde{X}$ be a projective variety, $D$\/ an effective ample divisor
on $\tilde{X}$, and $X = \tilde{X}\setminus D$. Also, by considering
an embedding $\tilde{X}\to \p^n$ associated to a suitable
multiple of~$D$, we can view $D$\/ as the intersection of $\tilde{X}$
with the hyperplane $H$ at infinity. By choosing an affine coordinate
system for $\p^n\setminus H$, we can consider the ring of regular
functions $\mathcal{O}[X]$. Note that the choice of the ring $\mathcal{O}[X]$
gives an integer model for $X$. We can define the set $X(\mathcal{O}_S)$
of the $S$--integral points of $X$ as the set of morphisms of algebras
$\mathcal{O}[X]\to \mathcal{O}_S$. There is a bijection between
this set and the points of $X$ which reduction modulo $p$ are not
in $D$.

For example we can see that Theorem \ref{sunit} implies that there are
only finitely many $S$--integral points in $\p^1\setminus\{0,\infty,1\}$.
Thus consider the ring of regular functions
$$
\mathcal{O}[\p^1\setminus\{0,\infty,1\}]
= \mathcal{O}\left[T,T^{-1},(T-1)^{-1}\right],
$$
to deduce that $S$--integral points of $\p^1\setminus\{0,\infty,1\}$
correspond to morphisms from
$\mathcal{O}_S\left[T,T^{-1},(T-1)^{-1}\right]$ to $\mathcal{O}_S$.
But such a morphism is the specialisation of~$T$\/
to an $S$--unit $u$ such that $1-u$ is an $S$--unit too.
Therefore if we write $v=1-u$ we obtain the equation $u+v=1$; by Theorem
\ref{sunit} it follows that there are only finitely many possible
values for the $S$--unit $u$.

As a direct application of the previous arguments we prove the
following result:

\begin{proposition}\label{sint}
Let $K$ and $S$ be as above.  Let $D$ be the effective divisor sum
of the lines and the conics defined in Lemma~$\ref{Lemm:M6S6}$.
Then the set of $S$--integral points of $M_2(6)=S_6\setminus D$
is finite.
\end{proposition}
\begin{proof}
We can consider $\p^1\times\p^1\times\p^1$ as the compactification
of $\A^3$ and consider the restriction to $M_2(6)$ of the rational
map $\Phi\colon \p^3\to\p^1\times\p^1\times\p^1$ obtained in the
canonical way from the map $\phi^{-1}$ defined in Lemma~$\ref{Lemm:M6S6}$.
The map $\Phi$  is an isomorphism from $M_2(6)$ to its image, which
is locally closed in $\p^1\times\p^1\times\p^1$. By Lemma \ref{Lemm:M6S6}
we see that each $S$--integral point is sent via the map $\Phi$
into a point $(x,y,z)\in \p^1\times\p^1\times\p^1$ where $x,y,z$
% are $S$--integral points in $\p^1\setminus\{[0:1],[1:0],[1:1]\}$.
are $S$--integral points in $\p^1\setminus\{0,\infty,1\}$.
By the argument described before the present proposition,
there are finitely many such $S$--integral points. Now the proposition
follows from the fact that $\Phi$ is a one-to-one map. \end{proof}

\begin{remark}
Proposition \ref{sint} also follows from \cite[Theorem 1.2]{Ca}.
An $S$--integral point of $M_2(6)=S_6\setminus D$ corresponds to
an ($n+1$)-tuple
$$
(f,0,\infty,1,x,y,z) = (f,[0:1],[1:0],[1:1],[x:1],[y:1],[z:1]),
$$
where $x,y,z$ are $S$--units and $f$ is a quadratic map defined over~$K$\/
with good reduction outside~$S$. See \cite{Sil.2} or \cite{Ca} for the
definition of good reduction, but roughly speaking it means that
the homogeneous resultant of the two $p$--coprime polynomials defining
$f$ is a $p$--unit for any $p\notin S$.  In particular, to an $S$--integral
point of $M_2(6)$ corresponds a rational map $f$\/ defined over~$K$,
with good reduction outside~$S$, which admits a $K$--rational periodic
point of minimal period~6; and this set is finite by \cite[Theorem 1.2]{Ca}.
Now Proposition \ref{sint} follows from the previous
argument because for any point $[W:X:Y:Z]\in M_2(6)$ there exists
a unique $f$\/ that admits the cycle $([0:1],[1:0],[1:1],[x:1],[y:1],[z:1])$,
where $(x,y,z)=\phi^{-1} ([W:X:Y:Z])$ and the map $\phi^{-1}$
is the one defined in Lemma~\ref{Lemm:M6S6}.
\end{remark}

% Apart from the elliptic curves described above, it is possible
% to find a lot of "isolated" rational points. Here is the list of
% all points $[W:X:Y:Z]\in \M_2(6)$, where $W,X,Y,Z\in \mathbb{Z}$
% are integers between $-200$ and $200$, satisfying that $XYZ\neq 0$:
% $$\begin{array}{lll}
% { }[\pm 16: -4: -5: 8]& [\pm 16: -5: 8: -4] & [\pm 16: 8: -4: -5]\\
% { } [\pm 20: -7: -5: 15]& [\pm 20: -5: 15: -7]& [\pm 20: 15: -7: -5]\\
% { }[\pm 16: -11: -1: 51]& [\pm 16: -1: 51: -11]& [\pm 16: 51: -11: -1]\\
%{ } [\pm 11: -59: -103: 139] & [\pm 11: -103: 139: -59]& [\pm 11: 139: -59: -103]\\
%   { }[\pm 127: -123: 93: 150]& [\pm 127: 93: 150: -123]& [\pm 127: 150: -123: 93]\\
%   { }[\pm 59: 91 : -52: 169]& [\pm 25: 91: 169: -52]& [\pm 59: 169: -52: 91]\\
%   { }[\pm 193: -95:-115: 155]& [\pm 193: -115:155:-95]& [\pm 193: 155 : -95: -115]\\
%   { }[\pm 25: 47: 103: 217]& [\pm 25: 103: 217: 47]& [\pm 25: 217: 47: 103]\\
%   { }[\pm 1088: -68:-85:88 ]& [\pm 1088: -85:88: -68]& [\pm 1088: 88: -68: -85]\\
%   \end{array}$$
% 1088     88    -68    -85
% 25   217    47   103
% 59   169   -52    91
% 135   153   -85     0
% 193   155   -95  -115

\end{document}